\documentclass[11pt]{amsart}
\usepackage{a4wide,amssymb,color}
\usepackage{relsize}
\usepackage[all]{xy}
\usepackage{bbm}
\newcommand{\korin}[2]{\!\sqrt[#1]{\!#2}}

\newcommand*{\defeq}{\stackrel{\mathsmaller{\mathsf{def}}}{=}}
\newcommand{\w}{\omega}

\newcommand{\IN}{\mathbb N}
\newcommand{\Ra}{\Rightarrow}

\newcommand{\two}{\mathbbm 2}
\newcommand{\IZ}{\mathbb Z}

\newtheorem{theorem}{Theorem}[section]
\newtheorem{proposition}[theorem]{Proposition}
\newtheorem{lemma}[theorem]{Lemma}
\newtheorem{corollary}[theorem]{Corollary}
\newtheorem{claim}[theorem]{Claim}

\newtheorem{example}[theorem]{Example}

\theoremstyle{definition}
\newtheorem{definition}[theorem]{Definition}
\newtheorem{remark}[theorem]{Remark}

\title{$E$-separated semigroups}
\author{Taras Banakh}
\address{Ivan Franko National University of Lviv, Universytetska 1, 79000, Lviv, Ukraine and Jan Kochanowski University in Kielce, Poland.}
\email{t.o.banakh@gmail.com}
\subjclass{20M10}
\keywords{$E$-central semigroup, the least semilattice congruence, the binary quasiorder}

\begin{document}
\begin{abstract} A semigroup is called {\em $E$-separated} if for any distinct idempotents $x,y\in X$ there exists a homomorphism $h:X\to Y$ to a semilattice $Y$ such that $h(x)\ne h(y)$. Developing results of Putcha and Weissglass, we characterize $E$-separated semigroups via certain commutativity properties of idempotents of $X$. Also we characterize $E$-separated semigroups in the class of $\pi$-regular $E$-semigroups.
\end{abstract}


\maketitle

\section{Introduction}

In this paper we introduce and study $E$-separated $E$-semigroups. A semigroup $X$ is defined to be {\em $E$-separated} if for any distinct idempotents $x,y\in X$ there exists a homomorphism $h:X\to Y$ to a semilattice $Y$ such that $h(x)\ne h(y)$. We recall that a {\em semilattice} is a commutative semigroup of idempotents. An element $x$ of a semigroup $X$ is an {\em idempotent} if $xx=x$. A semigroup $X$ is called an {\em $E$-semigroup} if the set $E(X)\defeq\{x\in X:xx=x\}$ is a subsemigroup of $X$. 

Developing results of Putcha and Weissglass \cite{PW}, in Theorem~\ref{t:separ}  we characterize $E$-separated semigroup via suitable commutativity properties of the idempotents of the semigroup. 

In Proposition~\ref{p:duo-Esep} we prove that the class of $E$-separated $E$-semigroups contains all duo semigroups (and hence all commutative semigroups). A semigroup $X$ is called  {\em  duo} if $xX=Xx$ for every $x\in X$. It is clear that each commutative semigroup is duo.
In Theorem~\ref{t:Esepar} we establish some structural properties of $E$-separated $E$-semigroups. In particular, we distinguish a natural subsemigroup ${\Updownarrow}E(X)$ of $X$ that admits homomorphic retractions onto the semilattice $E(X)$ and also on the Clifford part $H(X)\defeq\bigcup_{e\in E(X)}H_e$ of $X$. 

In Theorem~\ref{t:central} we characterize $E$-separated semigroups within the class of $\pi$-regular $E$-semigroups.
 
The main instrument for studying $E$-separated semigroups is the binary quasiorder whose properties are discussed in Section~\ref{s:binary}. 

\section{Preliminaries}\label{s:binary}

In this section we collect some standard notions that will be used in the paper. We refer to \cite{Howie} for Fundamentals of Semigroup Theory. 

We denote by $\w$  the set of all finite ordinals and by $\IN\defeq\w\setminus\{0\}$ the set of all positive integer numbers.

Let $X$ be a semigroup. For an element $x\in X$ let $$x^\IN\defeq\{x^n:n\in\IN\}$$ be the monogenic subsemigroup of $X$, generated by the element $x$. For two subsets $A,B\subseteq X$, let $AB\defeq\{ab:a\in A,\;b\in B\}$ be the product of $A,B$ in $X$. For a subset $A\subseteq X$ and number $n\in\IN$, let
$$\korin{n}{A}\defeq\{x\in X:x^n\in A\}\quad\mbox{and}\quad \korin{\infty}{A}\defeq\bigcup_{n\in\IN}\korin{n}{A}.$$

For an element $a$ of a semigroup $X$, the set
$$H_a=\{x\in X:(xX^1=aX^1)\;\wedge\;(X^1x=X^1a)\}$$
is called the {\em $\mathcal H$-class} of $a$.
Here $X^1=X\cup\{1\}$ where $1$ is an element such that $1x=x=x1$ for all $x\in X^1$.

By Corollary 2.2.6 \cite{Howie}, for every idempotent $e\in E(X)$ its $\mathcal H$-class $H_e$ coincides with the maximal subgroup of $X$, containing the idempotent $e$. The union $$H(X)=\bigcup_{e\in E(X)}H_e$$ of all maximal subgroups of $X$ is called the {\em Clifford part} of $X$ (it should be mentioned that $H(X)$ is not necessarily a subsemigroup of $X$).

For any element $x\in H(X)$, there exists a unique element $x^{-1}\in H(X)$ such that $$xx^{-1}x=x,\quad x^{-1}xx^{-1}=x^{-1},\quad\mbox{and}\quad xx^{-1}=x^{-1}x.$$

The set $$\korin{\infty}{H(X)}=\bigcup_{e\in E(X)}\korin{\infty}{H_e}$$is called the {\em eventually Clifford part} of $X$. Let $\pi:\korin{\infty}{H(X)}\to E(X)$ be the function assigning to each $x\in \korin{\infty}{H(X)}$ the unique idempotent $e\in E(X)$ such that $x^\IN\cap H_e\ne\emptyset$. The following lemma shows that the function $\pi$ is well-defined.

\begin{lemma}\label{l:pi-well-defined} Let $x$ be an element of a semigroup $X$ such that $x^n\in H_e$ for some $n\in\IN$ and $e\in E(X)$. Then $x^m\in H_e$ for all $m\ge n$.
\end{lemma}

\begin{proof} To derive a contradiction, assume that $x^m\notin H_e$ for some $m\ge n$. We can assume that $m$ is the smallest number such that $x^m\notin H_e$. It follows from $x^n\in H_e$ and $x^m\notin H_e$ that $m>n>1$ and hence $m-2\in\IN$. The minimality of $m$ ensures that $x^{m-1}\in H_e$. Observe that $x^{m}X^1\subseteq x^{m-1}X^1=ex^{m-1}X^1\subseteq eX^1$ and $$eX^1=x^{2(m-1)}(x^{2(m-1)})^{-1}X^1\subseteq x^{2(m-1)}X^1=x^{m}x^{m-2}X^1\subseteq x^{m}X^1.$$Therefore, $x^mX^1=eX^1$. By analogy one can prove that $X^1x^{m}=X^1e$. Therefore, $x^{m}\in H_e$, which contradicts the choice of $m$.
\end{proof}

A semigroup $X$ is called
\begin{itemize}
\item {\em Clifford} if $X=H(X)$;
\item {\em eventually Clifford} if $X=\korin{\infty}{H(X)}$.
\end{itemize}

In fact, the class of (eventially) Clifford semigroups coincides with the class of completely ($\pi$-)regular semigroups, considered in \cite{PR99} (and \cite{BC92},  \cite{Mitro2003}, \cite{PBC}).

Let us recall that a semigroup $X$ is defined to be 
\begin{itemize}
\item ({\em completely})  {\em regular} if for every $x\in X$ there exists $y\in X$ such that $x=xyx$ (and $xy=yx$);
\item ({\em completely}) {\em $\pi$-regular} if for every $x\in X$ there exist $n\in\IN$ and $y\in X$ such that $x^n=x^nyx^n$ (and $x^ny=yx^n$).
\end{itemize}

Each semilattice $X$ carries the {\em natural partial order} $\le$ defined by $x\le y$ off $xy=y=yx$.


Let $\two$ denote the set $\{0,1\}$ endowed with the operation of multiplication inherited from the ring $\IZ$. It is clear that $\two$ is a two-element semilattice, so it carries the natural partial order, which coincides with the linear order inherited from $\IZ$.

For elements $x,y$ of a semigroup $X$ we write $x\lesssim y$ if $\chi(x)\le \chi(y)$ for every homomorphism $\chi:X\to\two$. The relation $\lesssim$ is a quasiorder, called the {\em binary quasiorder} on $X$, see \cite{BH}. The obvious order properties of the semilattice $\two$ imply the following (obvious) properties of the binary quasiorder on $X$.

\begin{proposition}\label{p:quasi2} For any semigroup $X$ and any elements $x,y,a\in X$, the following statements hold:
\begin{enumerate}
\item if $x\lesssim y$, then $ax\lesssim ay$ and $xa\lesssim ya$;
\item $xy\lesssim yx\lesssim xy$;
\item $x\lesssim x^2\lesssim x$;
\item $xy\lesssim x$ and $xy\lesssim y$.
\end{enumerate}
\end{proposition}

For an element $a$ of a semigroup $X$ and subset $A\subseteq X$, consider the following sets:
$$
{\Uparrow}a\defeq\{x\in X:a\lesssim x\},\quad {\Downarrow}a\defeq\{x\in X:x\lesssim a\},\quad\mbox{and}\quad{\Updownarrow}a\defeq\{x\in X:a\lesssim x\lesssim a\},
$$
called the {\em upper $\two$-class}, {\em lower $\two$-class} and the {\em $\two$-class} of $x$, respectively. Proposition~\ref{p:quasi2} implies that those three classes are subsemigroups of $X$.

The following simple fact follows from the definition of the class ${\Updownarrow}x$.

\begin{proposition}\label{p:korin} For every idempotent $e$ of a semigroup $X$ we have $\korin{\infty}{H_e}\subseteq{\Updownarrow}e$.
\end{proposition}
 
For two elements $x,y$ of a semigroup $X$, we write $x\Updownarrow y$ iff ${\Updownarrow}x={\Updownarrow}y$ iff $\chi(x)=\chi(y)$ for any homomorphism $\chi:X\to\two$. Proposition~\ref{p:quasi2} implies that $\Updownarrow$ is a congruence on $X$. 

 We recall that a {\em congruence} on a semigroup $X$ is an equivalence relation $\approx$ on $X$ such that for any elements $x\approx y$ of $X$ and any $a\in X$ we have $ax\approx ay$ and $xa\approx ya$. For any congruence $\approx$ on a semigroup $X$, the quotient set $X/_\approx$ has a unique semigroup structure such that the quotient map $X\to X/_\approx$ is a semigroup homomorphism. The semigroup $X/_\approx$ is called the {\em quotient semigroup} of $X$ by the congruence $\approx$~.

A congruence $\approx$ on a semigroup $X$ is called a {\em semilattice congruence} if the quotient semigroup $X/_\approx$ is a semilattice. Proposition~\ref{p:quasi2} implies that $\Updownarrow$ is a semilattice congruence on $X$. Moreover, ${\Updownarrow}$ is equal to the smallest semilattice congruence on $X$, see \cite{BH}, \cite{Petrich63}, \cite{Petrich64}, \cite{Tamura73}. The quotient semigroup $X/_{\Updownarrow}$ is called the {\em semilattice reflexion} of $X$. More information on the smallest semilattice congruence and semilattice decompositions of semigroups can be found in \cite{Putcha73}, \cite{BCP}, \cite{Mitro2003}, \cite{Mitro2004}, \cite{Sulka}.

A semigroup $X$ is called {\em $\two$-trivial} if every homomorphism $h:X\to\two$  is constant. Tamura \cite{Tamura73}, \cite{Tamura82} calls $\two$-trivial semigroups {\em semilattice-indecomposable} (or briefy {\em $s$-indecomposable}) semigroups. The following fundamental fact was first proved by Tamura \cite{Tam56} and then reproved by another methods in \cite{TS66}, \cite{Petrich63}, \cite{Petrich64}, and \cite{BH}.

\begin{theorem}[Tamura]\label{t:Tamura} For every element $x$ of a semigroup $X$ its $\two$-class ${\Updownarrow}x$ is a $\two$-trivial semigroup.
\end{theorem}
 
The binary quasiorder admits an inner description via prime (co)ideals, which was first noticed by Petrich \cite{Petrich64} and Tamura \cite{Tamura73}.

A subset $I$ of a semigroup $X$ is called
\begin{itemize}
\item an {\em ideal} in $X$ if $(IX)\cup (XI)\subseteq I$;
\item a {\em prime ideal} if $I$ is an ideal such that $X\setminus I$ is a subsemigroup of $X$;
\item a ({\em prime}) {\em coideal} if the complement $X\setminus I$ is a (prime) ideal in $X$.
\end{itemize}
According to this definition, the sets $\emptyset$ and $X$ are prime (co)ideals in $X$.

Observe that a subset $A$ of a semigroup $X$ is a prime coideal in $X$ if and only if its {\em characteristic function} $$\chi_A:X\to\two,\quad \chi_A:x\mapsto\chi_A(x)\defeq\begin{cases}1&\mbox{if $x\in A$},\\
0&\mbox{otherwise},
\end{cases}
$$ is a homomorphism.  This function characterization of prime coideals implies the following inner description of the $\two$-quasiorder, first noticed by Tamura in \cite{Tamura73}.

\begin{proposition}\label{p:smallest-pi} For any element $x$ of a semigroup $X$, its upper $\two$-class ${\Uparrow}x$ coincides with the smallest coideal of $X$ that contains $x$. 
\end{proposition}

\begin{corollary} A semigroup $X$ is $\two$-trivial if and only if every nonempty prime ideal in $X$ coincides with $X$.
\end{corollary}

\begin{remark}\label{r:Bard} By \cite{ACMU}, \cite{CM} (see also \cite{BG1}, \cite{BG2}, \cite{Bard16}, \cite{Bard20}), $\two$-trivial semigroups can contain non-trivial ideals, in particular, there exist infinite congruence-free (and hence $\two$-trivial) monoids with zero.
\end{remark}

The following inner description of the upper $\two$-classes is a modified version of Theorem 3.3 in \cite{Petrich64}. Its proof can be found in \cite{BH}.

\begin{proposition}\label{p:Upclass} For any element $x$ of a semigroup $X$ its upper $\two$-class ${\Uparrow}x$ is equal to the union $\bigcup_{n\in\w}{\Uparrow}_{\!n}x$, where ${\Uparrow}_{\!0}x=\{x\}$ and $${\Uparrow}_{\!n{+}1}x\defeq\{y\in X:X^1yX^1\cap({\Uparrow}_{\!n}x)^2\ne \emptyset\}$$ for $n\in\w$. 
\end{proposition}  

For duo semigroups, Proposition~\ref{p:Upclass} simplifies to the following form, proved in \cite{BH}. 

\begin{proposition}\label{p:duo} For any element $a\in X$ of a duo semigroup $X$ we have 
$${\Uparrow}a=\{x\in X:a^\IN\cap XxX\ne\emptyset\}.$$
\end{proposition}  

A semigroup $X$ is called {\em Archimedean} if for any elements $x,y\in X$ there exists $n\in\IN$ such that $x^n\in XyX$ for some $a,b\in X$. A standard example of an Archimedean semigroup is the additive semigroup $\IN$ of positive integers. For commutative semigroups the following characterization (that can be easily derived from Proposition~\ref{p:duo}) was obtained by Tamura and Kimura in \cite{TK54}. 

\begin{theorem}\label{t:Archimed} A duo semigroup $X$ is $\two$-trivial if and only if $X$ is Archimedean.
\end{theorem}

For viable semigroups we have another simplification of Proposition~\ref{p:Upclass} due to Putcha and Weissglass \cite{PW}. Let us recall that a semigroup $X$ is called {\em viable} if for any $x,y\in X$ with $\{xy,yx\}\subseteq E(X)$ we have $xy=yx$. 

\begin{proposition}[Putcha--Weissglass]\label{p:PW} If $X$ is a viable semigroup, then for every idempotent $e\in E(X)$ we have ${\Uparrow}e=\{x\in X:e\in X^1xX^1\}$.
\end{proposition}

\begin{proof} Let ${\Uparrow}_{\!1}e\defeq \{x\in X:e\in X^1xX^1\}$. By Proposition~\ref{p:Upclass}, ${\Uparrow}_{\!1}e\subseteq {\Uparrow}e$. The reverse inclusion will follow from the minimality of the prime coideal ${\Uparrow}e$ as soon as we prove that ${\Uparrow}_{\!1}e$ is a prime coideal in $X$. It is clear from the definition that ${\Uparrow}_{\!1}e$ is a coideal. So, it remains to check that ${\Uparrow}_{\!1}e$ is a subsemigroup. Given any elements $x,y\in {\Uparrow}_{\!1}e$, find elements $a,b,c,d\in X^1$ such that $axb=e=cyd$. Then $axbe=ee=e$ and $(beax)(beax)=be(axbe)ax=beeax=beax$, which means that $beax$ is an idempotent. By the viability of $X$, $axbe=e=beax$. By analogy we can prove that $ecyd=e=ydec$. Then $aeaxydex=ee=e$ and hence $xy\in{\Uparrow}_{\!1}e$.
\end{proof}

Following Tamura \cite{Tamura82}, we define a semigroup $X$ to be {\em unipotent} if $X$ contains a unique idempotent. The following fundamental result was proved by Tamura \cite{Tamura82} and reproved by a different method in \cite{BH}.

\begin{theorem}[Tamura, 1982]\label{t:max-ideal} For the unique idempotent $e$ of an unipotent $\two$-trivial semigroup $X$, the maximal group $H_e$ of $e$ in $X$ is an ideal in $X$. 
\end{theorem}

An element of a semigroup $X$ is called {\em central} if it belongs to the {\em center}
$$Z(X)\defeq\{z\in X:\forall x\in X\;\;(zx=xz)\}$$of the semigroup $X$.
 
\begin{corollary}\label{c:EZK}The unique idempotent $e$ of a unipotent $\two$-trivial semigroup $X$ is central in $X$.
\end{corollary}

\begin{proof} Let $e$ be a unique idempotent of the unipotent semigroup $X$. By Tamura's Theorem~\ref{t:max-ideal}, the maximal subgroup $H_e$ of $e$ is an ideal in $X$. Then for every $x\in X$ we have $xe,ex\in H_e$. Taking into account that $xe$ and $ex$ are elements of the group $H_e$, we conclude that $ex=exe=xe$. This means that the idempotent $e$ is central in $X$. 
\end{proof}

For any idempotent $e$ of a semigroup $X$,
let 
$$\tfrac{H_e}e\defeq\{x\in X:xe=ex\in H_e\}.$$
The set $\frac{H_e}e$ is a subsemigroup of $X$. Indeed, for any $x,y\in\frac{H_e}e$ we have  $xye=xyee=x(ey)e=(xe)(ye)\in H_eH_e=H_e$ and $exy=eexy=e(xe)y=(ex)(ey)\in H_eH_e=H_e$, which means that $xy\in\frac{H_e}e$.

The following theorem nicely complements Theorem~\ref{t:max-ideal} and Corollary~\ref{c:EZK}.

\begin{theorem}\label{t:C-ideal} For any idempotent $e$ we have
$$\korin{\infty}{H_e}\subseteq\tfrac{H_e}e\subseteq{\Uparrow}e.$$
\end{theorem}

\begin{proof} Take any element $x\in \korin{\infty}{H_e}$. Since $x\in\korin{\infty}{H_e}$, there exists $n\in\IN$ such that $x^n\in H_e$ and hence $x^{2n}\in H_e$. Observe that $xeX^1=xx^{n}X^1\subseteq x^nX^1=eX^1$ and $eX^1=x^{2n}X^1\subseteq x^{n+1}X^1=xeX^1$ and hence $xeX^1=eX^1$. By analogy we can prove that $X^1xe=X^1e$. Then $xe\in H_e$ by the definition of the $\mathcal H$-class $H_e$.

By analogy we can prove that $ex\in H_e$. It follows from $xe,ex\in H_e$ that $ex=exe=ex\in H_e$ and hence $x\in\frac{H_e}e$.

By Proposition~\ref{p:Upclass},
$$\tfrac{H_e}e\subseteq\{x\in X:e\in xH_e\cap xH_e\}\subseteq\{x\in X:e\in X^1xX^1\}\subseteq{\Uparrow}e.$$
\end{proof} 

An idempotent $e$ of a semigroup $X$ is defined to be {\em viable} if the semigroup $\tfrac{H_e}e$ is a coideal in $X$.

\begin{proposition}\label{p:e-viable} An idempotent $e$ of a semigroup $X$ is viable if and only if $\frac{H_e}e={\Uparrow}e$. In this case $H_e$ is an ideal of the semigroup ${\Uparrow}e$ and $e\in Z({\Uparrow}e)$.
\end{proposition}

\begin{proof} If $e$ is viable, then semigroup $\frac{H_e}e$ is a prime coideal in $X$ and hence ${\Uparrow}e\subseteq\frac{H_e}e$ as  ${\Uparrow}e$ is the smallest prime coideal containing $e$, see Proposition~\ref{p:smallest-pi}. Then  $\frac{H_e}e={\Uparrow}e$ by Theorem~\ref{t:C-ideal}.

If $\frac{H_e}e={\Uparrow}e$, then $e$ is viable because ${\Uparrow}e=\frac{H_e}e$ is a coideal in $X$.

Also $H_e$ is an ideal in $\frac{H_e}e$ and $e\in Z(\frac{H_e}e)$ by the definition of $\frac{H_e}e$.
\end{proof}

\section{Characterizing $E$-separated semigroups}

In this section we find several commutativity properties of semigroups, which are equivalent to the $E$-separatedness.

\begin{definition}
A semigroup $X$ is defined to be 
\begin{itemize}
\item {\em $E$-commutative} if $xy=yx$ for any idempotents $x,y\in E(X)$;
\item {\em $E$-viable} if every idempotent of $X$ is viable;
\item {\em $E$-central} if for any $e\in E(X)$ and $x\in X$ we have $ex=xe$;
\item {\em $E_{\Uparrow}$-central} if for any $e\in E(X)$ and $x\in {\Uparrow}e$ we have $ex=xe$;
\item {\em $E$-hypercentral} if for any $e\in E(X)$ and $x,y\in X$ with $xy=e$ we have $xe=ex$ and $ye=ey$;
\item {\em $E$-hypocentral} if for any $e\in E(X)$ and $x,y\in X$ with $xy=e$ we have $xe=ex$ or $ye=ey$;
\item {\em $E$-upcentral} if for any idempotents $e,f\in E(X)$ with $fe=e=ef$ and any $x\in\korin{\infty}{H_f}$ we have $xe=ex$.
\end{itemize}
\end{definition}

For any semigroup  these commutativity properties relate as follows.
$$
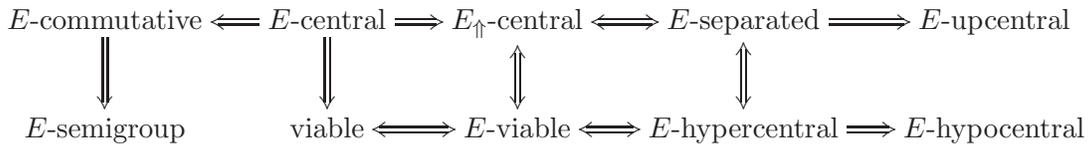
\xymatrix@C=18pt{
\mbox{$E$-commutative}\ar@{=>}[d]&\mbox{$E$-central}\ar@{=>}[r]\ar@{=>}[l]\ar@{=>}[d]&\mbox{$E_{\Uparrow}$-central}\ar@{<=>}[r]&\mbox{$E$-separated}\ar@{=>}[r]&\mbox{$E$-upcentral}\\
\mbox{$E$-semigroup}&\mbox{viable}\ar@{<=>}[r]&\mbox{$E$-viable}\ar@{<=>}[r]\ar@{<=>}[u]&\mbox{$E$-hypercentral}\ar@{<=>}[u]\ar@{=>}[r]&\mbox{$E$-hypocentral}
}
$$

Nontrivial equivalences and implications in this diagram are proved in the following theorem.

\begin{theorem}\label{t:separ} For a semigroup $X$ the following conditions are equivalent:
\begin{enumerate}
\item $X$ is $E$-separated;
\item $X$ is $E$-viable;
\item $X$ is $E_{\Uparrow}$-central;
\item $X$ is $E$-hypercentral;
\item $X$ is viable.
\end{enumerate}
The equivalent conditions \textup{(1)--(5)} imply the condition
\begin{enumerate}
\item[(6)] $X$ is $E$-hypocentral and $E$-upcentral.
\end{enumerate}
\end{theorem}

\begin{proof} We shall prove the implications $(1)\Ra(2)\Ra(3)\Ra(4)\Ra(5)\Ra(1)$ and $(4)\Ra(6)$.
\smallskip

$(1)\Ra(2)$ Assume that $X$ is $E$-separated. To show that $X$ is $E$-viable, take any $e\in E(X)$ and $x\in{\Uparrow}e$. Since $X$ is $E$-separated, the $\two$-class ${\Updownarrow}e$ of $e$ is unipotent. By Tamura's Theorem~\ref{t:max-ideal}, the group $H_e$ is an ideal in ${\Updownarrow}e$. Since ${\Updownarrow}e$ is an ideal in ${\Uparrow}e$, the maximal subgroup $H_e$ is an ideal in $X$. Then $xe,ex\in H_e$ and hence $xe=exe=ex\in H_e$ and $x\in\frac{H_e}e$. So ${\Uparrow}e\subseteq\frac{H_e}e$ and ${\Uparrow}e=\frac{H_e}e$ by Theorem~\ref{t:C-ideal}. Then $\frac{H_e}e={\Uparrow}e$ is a coideal in $X$ and the idempotent $e$ is viable, witnessing that the semigroup $X$ is $E$-viable.
\smallskip

The implication $(2)\Ra(3)$ follows from Proposition~\ref{p:e-viable}.
\smallskip

$(3)\Ra(4)$  Assume that $X$ is $E_{\Uparrow}$-central. To show that $X$ is $E$-hypercentral, take any idempotent $e\in E(X)$ and any elements $x,y\in X$ with $xy=e$. Proposition~\ref{p:quasi2} ensures that $e\lesssim x$ and $e\lesssim y$ and hence $x,y\in{\Uparrow}e$. Applying the $E_{\Uparrow}$-centrality of $X$, we conclude that $ex=xe$ and $ey=ye$.
\smallskip

$(4)\Ra(5)$ Assume that $X$ is $E$-hypercentral. To show that $X$ is viable, take any elements $x,y\in X$ such that $\{xy,yx\}\subseteq E(X)$. The $E$-hypercentrality of $X$ ensures that $xy=xyxy=x(yx)y=(yx)xy=yx(xy)=y(xy)x=yxyx=yx$.
\smallskip

$(5)\Ra(1)$ To derive a contradiction, assume that $X$ is viable but not $E$-separated. Then there exist two distinct idempotents $e,f\in E(X)$ such that ${\Uparrow}e={\Uparrow}f$. By Proposition~\ref{p:PW}, there are elements $a,b,c,d\in X^1$ such that $e=afb$ and $f=ced$. Observe that $afbe=ee=e$ and $(beaf)(beaf)=be(afbe)af=beeaf=beaf$ and hence $afbe$ and $beaf$ are idempotents. The viability of $X$ ensures that $afbe=beaf$. By analogy we can prove that $eafb=e=efbea$, $cedf=f=dfce$ and $fced=f=edfc$. These equalities imply that $H_e=H_f$ and hence $e=f$ because the group $H_e=H_f$ contains a unique idempotent. But the equality $e=f$ contradicts the choice of the idempotents $e,f$.
\smallskip

$(4)\Ra(6)$ Assume that $X$ is $E$-hypercentral. Then $X$ is $E$-hypocentral. To show that $X$ is $E$-upcentral, take any idempotents $e,f\in E(X)$ and any element $x\in\korin{\infty}{H_f}$ such that $fe=e=ef$. By Lemma~\ref{l:pi-well-defined}, there exists a number $n\ge 2$ such that $x^n\in H_f$. Let $g$ be the inverse element to $x^n$ in the group $H_f$. Then $e=fe=x^nge=x(x^{n-1}ge)$. The $E$-hypercentrality of $X$ ensures $ex=xe$.
\end{proof}

\begin{remark} Viable semigroups were introduced and studied by Putcha and Weissglass who proved in \cite[Theorem 6]{PW} that a semigroup $X$ is viable if and only if it is $E$-separated (this is the equivalence $(1)\Leftrightarrow(5)$ in Theorem~\ref{t:separ}). For another condition (involving $\mathcal J$-classes), equivalent to the conditions (1)--(5) of Theorem~\ref{t:separ}, see Theorem 23.7 in \cite{MS}. 
\end{remark}

\begin{example}\rm Any semigroup $X$ with left zero multiplication $xy=x$ is $E$-hypocentral and $E$-upcentral. If $X$ contains more than one element, then $X$ is not $E$-hypercentral. This example shows that condition (6) of Theorem~\ref{t:separ} is not equivalent to conditions (1)--(5). 
\end{example}

\begin{remark} By \cite{ACMU}, \cite{CM}, there exists an infinite $0$-simple congruence-free monoid $X$. Being
congruence-free, the semigroup $X$ is $\two$-trivial. On the other hand, $X$ contains at least two central idempotents: $0$ and $1$.  The polycyclic monoids (see \cite{Bard16}, \cite{Bard20}, \cite{BG1}, \cite{BG2}) have the similar properties. By Theorem 2.4 in \cite{Bard16}, for any cardinal $\lambda\ge 2$ the polycyclic monoid $P_\lambda$ is
congruence-free and hence $\two$-trivial, but its contains two distinct central idempotents $0$ and $1$. These examples show that individual central idempotents are not necessarily viable. On the other hand, if all idempotents of a semigroup are central, then all of them are viable, by Theorem~\ref{t:separ}.
\end{remark}

\section{ $E$-separated $E$-semigroups}

In this section we establish some results on the structure of $E$-separated $E$-semigroups. But first we show that the class of such semigroups contains all duo semigroups and hence all commutative semigroups.

\begin{proposition}\label{p:duo-Esep} Each duo semigroup $X$ is an $E$-separated $E$-semigroup.
\end{proposition}

\begin{proof} First we show that $X$ is an $E$-semigroup. Given two idempotents $e,f$, use the duo property of $X$ to find  elements $x,y\in X$ such that $ef=xe$ and $fe=yf$. Then $efef=eyff=eyf=efe=xee=xe=ef$ and hence $ef$ is an idempotent. Therefore, $X$ is an $E$-semigroup. 

Assuming that $X$ is not $E$-separated, we can find an idempotent $e\in E(X)$ whose $\two$-class ${\Updownarrow}e$ contains an idempotent $f\ne e$. By Proposition~\ref{p:duo}, $e\in XfX=Xf=fX$ and $f\in XeX=Xe=eX$. Then $eX^1\subseteq fXX^{1}\subseteq fX^1$, $fX^1\subseteq eXX^1\subseteq eX^1$, $X^1e=X^1Xf\subseteq X^1f$, and $X^1f=X^1Xe\subseteq X^1e$, which implies  $H_f=H_e$ and hence $f=e$ as the group $H_e=H_f$ contains a unique idempotent.
\end{proof}

The following theorem describing properties of $E$-separated $E$-semigroups is the main result of this section. The statements (2), (3) of this theorem holds true for any $E$-separated semigroup.
 
\begin{theorem}\label{t:Esepar} Any $E$-separated $E$-semigroup $X$ has the following properties.
\begin{enumerate}
\item $E(X)$ is a semilattice.
\item For any idempotent $e\in E(X)$ the maximal subgroup $H_e\subseteq X$ is an ideal in the semigroup ${\Uparrow}e$.
\item For any $e\in E(X)$ and $x\in{\Uparrow}e$ we have $ex=xe\in H_e$;
\item For any idempotents $x,y\in E(X)$, the inequality $x\lesssim y$ in $X$ is equivalent to the inequality $x\le y$ in $E(X)$.
\item The map $\pi_{\Updownarrow}:{\Updownarrow}E(X)\to E(X)$ assigning to each element $x\in {\Updownarrow}E(X)$ the unique idempotent in the semigroup ${\Updownarrow}x$ is a well-defined homomorphic retraction of the semigroup ${\Updownarrow}E(X)$ onto $E(X)$.
\item The map $\hbar_{\Updownarrow}:{\Updownarrow}E(X)\to H(X)$, $\hbar_{\Updownarrow}:x\mapsto x\pi_{\Updownarrow}(x)$, is a well-defined homomorphic retraction of the semigroup ${\Uparrow}E(X)$ onto the Clifford part $H(X)$ of $X$.
\item The Clifford part $H(X)$ is a subsemigroup of $X$.
\end{enumerate}
\end{theorem}

\begin{proof} Let $X/_{\Updownarrow}$ be the semilattice reflexion of $X$ and $q:X\to X/_{\Updownarrow}$ be the quotient homomorphism.
\smallskip

1. To see that $E(X)$ is a semilattice, take any idempotents $x,y\in E(X)$. Since $X$ is an $E$-semigroup, the products $xy$ and $yx$ are idempotents. Taking into account that $q:X\to X/_{\Updownarrow}$ is a homomorphism onto the semilattice $X/_{\Updownarrow}$, we conclude that $$q(xy)=q(x)q(y)=q(y)q(x)=q(yx)$$ and hence ${\Updownarrow}xy={\Updownarrow}yx$. Since the semigroup $X$ is $E$-separated,  the idempotents $xy$ and $yx$ are equal to the unique idempotent of the unipotent semigroup ${\Updownarrow}xy={\Updownarrow}yx$ and hence $xy=yx$.
\smallskip

2,3. The statements 2 and 3 follows from Theorem~\ref{t:separ} and Proposition~\ref{p:e-viable}.
%
\smallskip

4. Let $x,y$ be two idempotents in $X$. If $x\le y$, then $x=xy$ and hence $h(x)=h(x)h(y)\le h(y)$ for any homomorphism $h:X\to\two$. Then $x\lesssim y$ by the definition of the quasiorder $\lesssim$. Now assume that $x\lesssim y$. Multiplying this inequality by $x$ from both sides and applying Proposition~\ref{p:quasi2}, we obtain $x=xx\lesssim xy\lesssim x$ and hence $xy\in{\Updownarrow}x$. Since $X$ is an $E$-semigroup, the product $xy$ is an idempotent. Since the semigroup ${\Updownarrow}x$ is unipotent, the idempotent $xy\in{\Updownarrow}x$ is equal to the unique idempotent $x$ of ${\Updownarrow}x$. The equality $x=xy$ means $x\le y$, by the definition of the partial order $\le$ on the semilattice $E(X)$.
\smallskip

5. Consider the map $\pi_{\Updownarrow}:{\Updownarrow}E(X)\to E(X)$ assigning to each element $x\in X$ the unique idempotent in the unipotent semigroup ${\Updownarrow}x$. It is clear that $\pi_{\Updownarrow}$ is a retraction of ${\Updownarrow}E(X)$ onto $E(X)$. Since ${\Updownarrow}$ is a semilattice congruence, the quotient semigroup $X/_{\Updownarrow}$ is a semilattice and the quotient map $q:X\to X/_{\Updownarrow}$ is a semigroup homomorphism. By the $\Updownarrow$-unipotence of $X$, the restriction $h\defeq q{\restriction}_{E(X)}:E(X)\to q[E(X)]\subseteq X/_{\Updownarrow}$ is bijective and hence $h$ is a semigroup isomorphism and so is the inverse function $h^{-1}:q[E(X)]\to E(X)$. Then the function $\pi=h^{-1}\circ q{\restriction}_{{\Updownarrow}E(X)}$ is a semigroup homomorphism, being a composition of two homomorphisms.
\smallskip

6. Since the function $\pi_{\Updownarrow}:{\Updownarrow}E(X)\to E(X)$ is well-defined, so is the function $\hbar_{\Updownarrow}:{\Updownarrow}E(X)\to X$, $\hbar_{\Updownarrow}:x\mapsto x\pi_{\Updownarrow}(x)$. To see that $\hbar_{\Updownarrow}$ is a homomorphism, take any elements $x,y\in {\Updownarrow}E(X)$ and applying Theorem~\ref{t:Esepar}(5,3), conclude that
$$\hbar_{\Updownarrow}(xy)=xy\pi_{\Updownarrow}(xy)=xy\pi_{\Updownarrow}(x)\pi_{\Updownarrow}(y)=x\pi_{\Updownarrow}(x)\pi_{\Updownarrow}(y)y=x\pi_{\Updownarrow}(x)y\pi_\Updownarrow(y)=\hbar_{\Updownarrow}(x)\hbar_{\Updownarrow}(y).$$
By Theorem~\ref{t:max-ideal}, for any $e\in E(X)$ and $x\in{\Updownarrow}e$, the group $H_{e}$ is an ideal in ${\Updownarrow}e$ and hence $\hbar_{\Updownarrow}(x)=x\pi_{\Updownarrow}(x)=xe\in H_{e}\subseteq H(X)$. If $x\in H(X)$, then $x\in H_{e}$, and hence $\hbar_{\Updownarrow}(x)=xe=x$. Therefore, $\hbar_{\Updownarrow}:{\Updownarrow}E(X)\to H(X)$ is a well-defined homomorphic retraction of ${\Updownarrow}E(X)$ onto $H(X)$.
\smallskip

7. Since $\hbar_{\Updownarrow}:{\Updownarrow}E(X)\to X$ is a homomorphism, its image $H(X)=\hbar_{\Updownarrow}[{\Updownarrow}E(X)]$ is a subsemigroup of $X$. 
\end{proof}

\section{Characterizing $E$-separated $\pi$-regular $E$-semigroups}

In this section we recognize $E$-separated semigroups among $\pi$-regular $E$-semigroups. We recall that a semigroup $X$  is {\em $\pi$-regular} if for every $x\in X$ there exist $n\in\IN$ and $y\in X$ such that $x^n=x^nyx^n$. The class of $\pi$-regular semigroups includes all eventually Clifford semigroups (called also completely $\pi$-regular semigroups). A semigroup $X$ is {\em eventually Clifford} if $X=\korin{\infty}{H(X)}$. For any semigroup $X$ by $\pi:\korin{\infty}{H(X)}\to E(X)$ we denote the function assigning to each $x\in \korin{\infty}{H(X)}$ the unique idempotent $e\in E(X)$ such that $x^\IN\cap H_e\ne\emptyset$.

\begin{proposition}\label{p:Eup} If a semigroup $X$ is $E$-commutative and $E$-upcentral, then
\begin{enumerate}
\item for every $e,f\in E(X)$ we have $H_eH_f\subseteq H_{ef}$;
\item for every idempotents $e,f\in E(X)$ with $e\le f$ we have $(\korin{\infty}{H_f}\cdot H_e)\cup(H_e\cdot\korin{\infty}{H_f})\subseteq H_{e}$;
\item for every idempotents $e,f\in E(X)$ and every elements $x\in\korin{\infty}{H_e}$ and $y\in\korin{\infty}{H_f}$ we have $(xy)^nef\in H_{ef}$ for all $n\in\IN$;
\item for any $x,y\in\korin{\infty}{H(X)}$ with $xy\in\korin{\infty}{H(X)}$ we have $\pi(x)\pi(y)\le\pi(xy)$;
\item for any $e\in E(X)$ and $x\in X$ with $\{xe,ex\}\subseteq\korin{\infty}{H(X)}$, we have $\pi(xe)=\pi(ex)$;
\item for any $e\in E(X)$ and $x\in\korin{\infty}{H(X)}$ with $xe\in\korin{\infty}{H(X)}$ we have $\pi(xe)=\pi(x)e$.
\end{enumerate}
\end{proposition} 

\begin{proof} 1.   Let $u\in H_e$ and $v\in H_f$. The $E$-upcentrality of $X$ ensures that $efu=uef$ and $efv^{-1}=v^{-1}ef$. Then $efuv=uefv=uevf$ and $uvv^{-1}u^{-1}=ufu^{-1}=uefu^{-1}=efuu^{-1}=efe=ef$ and $v^{-1}u^{-1}uv=v^{-1}ev=v^{-1}efv=efv^{-1}v=eff=ef$. Hence $uv\in H_{ef}$, witnessing that $H_eH_f\subseteq H_{ef}$.
\smallskip

2. For every $e,f\in E(X)$ with $e\le f$ and every $x\in \korin{\infty}{H_f}$, we have
$$xe=xfe\in \korin{\infty}{H_f}fe\subseteq H_{f}e\subseteq H_{fe}=H_e,$$see Theorem~\ref{t:C-ideal} and Proposition~\ref{p:Eup}(1).  By analogy we can prove that $ex\in H_e$.
\smallskip

3. Let $e,f\in E(X)$ and $x\in\korin{\infty}{H_e}$, $y\in\korin{\infty}{H_f}$ be any elements. By induction we shall prove that $(xy)^nef\in H_{ef}$ for every $n\in\IN$. For $n=1$ we have
$$
xyef=xefy\in \textstyle \big(\!\korin{\infty}{H_{e}}\cdot H_{e}\big)\cdot \big(H_{f}\cdot\korin{\infty}{H_{f}}\big)\subseteq H_{e}H_f\subseteq H_{ef}
$$
by the $E$-upcentrality of $X$, Theorem~\ref{t:C-ideal} and Proposition~\ref{p:Eup}(1). 
Assume that for some $n\in\IN$ we have proved that $(xy)^nef\in H_{ef}$. Then
$(xy)^{n+1}ef=(xy)^nxyef\in H_{ef}H_{ef}=H_{ef}$ by the inductive assumption and case $n=1$.
\smallskip

4.  Take any elements $x,y\in \korin{\infty}{H(X)}$ with $xy\in\korin{\infty}{H(X)}$. Since $xy\in \korin{\infty}{H(X)}$, there exists $n\in\IN$ such that $(xy)^n\in H_{\pi(xy)}$. By Proposition~\ref{p:Eup}(1),
$$(xy)^n\pi(x)\pi(y)\in H_{\pi(xy)}H_{\pi(x)}H_{\pi(y)}\subseteq H_{\pi(xy)\pi(x)\pi(y)}.$$
On the other hand, Proposition~\ref{p:Eup}(3) ensures that $$(xy)^n\pi(x)\pi(y)\in H_{\pi(x)\pi(y)}.$$ Hence $\pi(xy)\pi(x)\pi(y)=\pi(x)\pi(y)$, which means that $\pi(x)\pi(y)\le\pi(xy)$.
\smallskip

5. Take any elements $e\in E(X)$ and $x\in X$ such that $\{xe,ex\}\subseteq\korin{\infty}{H(X)}$. By Lemma~\ref{l:pi-well-defined}, there exists $n\in\IN$ such that $(xe)^n\in H_{\pi(xe)}$ and $(ex)^n\in H_{\pi(ex)}$. Then 
\begin{multline*}
$$H_{\pi(xe)}\ni (xe)^{n+1}=x(ex)^ne=x(ex)^n\pi(ex)e=x(ex)^ne\pi(ex)=\\
(xe)^{n+1}\pi(ex)\in H_{\pi(xe)}\pi(ex)\subseteq H_{\pi(xe)\pi(ex)}
\end{multline*} and hence $\pi(xe)=\pi(xe)\cdot\pi(ex)$. By analogy we can prove that $\pi(ex)=\pi(ex)\cdot\pi(xe)$. Then $\pi(xe)=\pi(xe)\pi(ex)=\pi(ex)\pi(xe)=\pi(ex)$.


\smallskip

6. Take any $e\in E(X)$ and $x\in \korin{\infty}{H(X)}$ with $xe\in\korin{\infty}{H(X)}$. Find $n\in\IN$ such that $\{(xe)^n,x^n\}\subseteq H(X)$.  Let $f\defeq\pi(xe)$ and observe that $H_f\ni (xe)^n=(xe)^ne\subseteq H_fe\subseteq H_{fe}$ implies $f=fe$.

By induction we shall prove that $(xf)^k=(xe)^k f$. For $k=1$ this follows from $f=ef$. Assume that for some $k\in\IN$ we have $(xf)^k=(xe)^kf$.
By the inductive assumption and Theorem~\ref{t:C-ideal}, 
$$(xf)^{k+1}=(xf)^kxf=(xe)^kfxef=(xe)^k\pi(xe)xef=(xe)^kxe\pi(xe)f=(xe)^{k+1}ff=(xe)^{k+1}f.$$ This complete the inductive step and also the proof of the equality $(xf)^k=(xe)^kf$ for all $k\in\IN$. 

For $k=n$ we obtain
$$(xf)^n=(xe)^nf\in H_{\pi(xe)}f\subseteq H_{\pi(xe)f}=H_f,$$which implies $xf\in\korin{\infty}{H_f}$ and $\pi(xf)=f$.

By induction we shall prove that $(xf)^k=x^k f$. For $k=1$ this is trivial. Assume that for some $k\in\IN$ we have proved that $(xf)^k=x^kf$.
By the inductive assumption and Theorem~\ref{t:C-ideal},
$$(xf)^{k+1}=(xf)^kxf=x^kfxf=x^k\pi(xf)xf=x^kxf\pi(xf)=x^{k+1}ff=x^{k+1}f.$$ This complete the inductive step and also the proof of the equality $(xf)^k=x^kf$ for all $k\in\IN$. 

The choice of $n$ ensures that $x^n\in H(X)$ and hence $x^n\in H_{\pi(x)}$ and $x^n=x^n\pi(x)$. By Proposition~\ref{p:Eup}(4), $\pi(x)e\le \pi(xe)=f$ and hence $\pi(x)e=\pi(x)ef$. Then
$$H_{\pi(x)e}\ni x^ne=x^n\pi(x)e=x^n(\pi(x)ef)=(x^n\pi(x))fe=x^nfe=(xf)^ne\in H_fe\subseteq H_{fe}$$and finally, $\pi(x)e=fe=f=\pi(xe)$.
\end{proof}

Now we are able to prove the main result of this section.

\begin{theorem}\label{t:central} For a $\pi$-regular $E$-semigroup $X$, the following conditions are equivalent:
\begin{enumerate}
\item ${\Updownarrow}{e}=\korin{\infty}{H_e}$ for every $e\in E(X)$;
\item $X$ is $E$-separated;
\item $X$ is $E$-upcentral, $E$-hypocentral, and $E$-commutative.
\end{enumerate}
\end{theorem}

\begin{proof} We shall prove the implications $(1)\Ra(2)\Ra(3)\Ra(1)$. Let $q:X\to X/_{\Updownarrow}$, $q:x\mapsto {\Updownarrow}x$, be the quotient homomorphism of $X$ onto its semilattice reflexion.
\smallskip

$(1)\Ra(2)$ If ${\Updownarrow}{e}=\korin{\infty}{H_e}$ for every $e\in E(X)$, then for every distinct idempotents $e,f\in E(X)$ we have $$q(e)={\Updownarrow}e=\korin{\infty}{H_e}\ne\korin{\infty}{H_f}={\Updownarrow}f=q(f),$$which means that the semigroup $X$ is $E$-separated.
\smallskip

$(2)\Ra(3)$ If $X$ is $E$-separated, then $X$ is $E_{\Uparrow}$-central and $E$-hypocentral by Theorem~\ref{t:separ}. To see that $X$ is $E$-commutative, take any idempotents $x,y\in E(X)$. Since $X$ is an $E$-semigroup, the products $xy,yx$ are idempotents. By Theorem~\ref{t:separ}, the $E$-separated semigroup $X$ is viable and hence $xy=yx$.
\smallskip

$(3)\Ra(1)$ Assume that a $\pi$-regular semigroup $X$ is $E$-upcentral, $E$-hypocentral, and $E$-commutative. 

\begin{claim} The semigroup $X$ is eventually Clifford.
\end{claim}

\begin{proof} Take any $x\in X$ and using the $\pi$-regularity of $X$, find $n\in\IN$ and $y\in X$ such that $x^n=x^nyx^n$. It follows that $e=x^ny$ and $f=yx^n$ are idempotents. Since $X$ is $E$-hypocentral, $e=x^ny$ implies $x^ne=ex^n$ or $ey=ye$. If $x^ne=ex^n$, then
$f=ff=(yx^n)(yx^n)=y(x^ny)x^n=yex^n=yx^ne=fe$. If $ey=ye$, then
$f=ff=(yx^n)(yx^n)=y(x^ny)x^n=yex^n=eyx^n=ef=fe$. In both cases we obtain $f=fe$.

On the other hand, by the $E$-hypocentrality of $X$, the equality $f=yx^n$ implies $fy=yf$ or $fx^n=x^nf$. If $fy=yf$, then
$e=ee=x^nyx^ny=x^nfy=x^nyf=ef$. If $fx^n=x^nf$, then $e=ee=x^nyx^ny=x^nfy=fx^ny=fe=ef$. In both cases we obtain $e=ef$. Therefore, $e=ef=f$.

Observe that $eX^1=x^nyX^1\subseteq x^nX^1$ and $x^nX^1=x^nyx^nX^1=ex^nX^1\subseteq eX^1$ and hence $eX^1=x^nX^1$. On the other hand, $X^1x^n=X^1x^nyx^n\subseteq X^1yx^n=X^1f=X^1e$ and $X^1e=X^1f=X^1yx^n\subseteq X^1x^n$ and hence $X^1e=X^1x^n$. The equalities $eX^1=x^nX^1$ and $X^1e=X^1x^n$ imply $x^n\in H_e$. Then $x\in\korin{\infty}{H_e}\subseteq\korin{\infty}{H(X)}$.
\end{proof}

Since the semigroup $X$ is eventually Clifford, the map $\pi:X\to E(X)$ is well-defined on the whole semigroup $X=\korin{\infty}{H(X)}$.

\begin{claim}\label{cl:Up} For every $e\in E(X)$, the upper $\two$-set ${\Uparrow}e$ is equal to the set $${\Uparrow}_{\!\pi} e\defeq\{x\in X:e\le\pi(x)\}.$$
\end{claim}

\begin{proof} Given any $x\in{\Uparrow}_{\!\pi} e$, find $n\in\IN$ such that $x^n\in H_{\pi(x)}$ and conclude that $e\le\pi(x)\Updownarrow x$ implies $x\in{\Uparrow}e$, by Proposition~\ref{p:quasi2}. Therefore, ${\Uparrow}_{\!\pi} e\subseteq{\Uparrow}e$. The equality ${\Uparrow}_\pi e={\Uparrow}e$ will follow from the minimality of the prime coideal ${\Uparrow}e$ as soon as we check that the set ${\Uparrow}_{\!\pi} e$ is a prime coideal in $X$.

By Proposition~\ref{p:Eup}(4), for every $x,y\in {\Uparrow}_{\!\pi} e$ we have $e=ee\le\pi(x)\pi(y)\le\pi(xy)$ and hence $xy\in {\Uparrow}_{\!\pi} e$ and ${\Uparrow}_{\!\pi} e$ is a semigroup. Next, we show that $I\defeq X\setminus {\Uparrow}_{\!\pi} e$ is an ideal in $X$. Assuming that $I$ is not an ideal, we can find elements $x\in I$ and $y\in X$ such that $xy$ or $yx$ belongs to $X\setminus I={\Uparrow}_{\!\pi} e$. First we consider the case $xy\in {\Uparrow}_{\!\pi} e$. By Theorem~\ref{t:Esepar}(2) , $exy\in H_e$ and hence there exists an element $g\in H_e$ such that $exyg=e$. Assuming that $ex\in {\Uparrow}_{\!\pi}e$ and applying Proposition~\ref{p:Eup}(6), we conclude that $e\le \pi(ex)=\pi(eex)=e\pi(ex)\le e$ and hence $e=\pi(ex)$. Applying Proposition~\ref{p:Eup}(6) once more, we conclude that $e=\pi(ex)=e\pi(x)\le\pi(x)$ and $x\in{\Uparrow}_{\!\pi}e$, which contradicts the choice of $x$. Therefore, $ex\notin {\Uparrow}_{\!\pi}e$. Replacing the elements $x,y$ by $ex$ and $yg$, we can assume that $ex=x$, $ye=y$ and $xy=e$. Consider the product $f=yx$ and observe that $ff=yxyx=yex=yx=f$, which means that $f$ is an idempotent. By the $E$-hypocentrality of $X$, the equality $xy=e$ implies $xe=ex$ or $ye=ey$. If $xe=ex$, then $f=yx=yex=yxe=fe$. If $ye=ey$, then $f=yx=yex=eyx=ef=fe$. In both cases we conclude that $f=fe$. By the $E$-hypocentrality of $X$, the equality $f=yx$ implies $fy=yf$ or $fx=xf$. If $fy=yf$, then $f=ef=xyf=xfy=xyxy=ee=e$. If $fx=xf$, then $f=fe=fxy=xfy=xyxy=e$. In both cases we obtain $e=f$.

Now observe that $xX^1=exX^1\subseteq eX^1$ and $eX^1=xyX^{-1}\subseteq xX^1$, which implies $xX^1=eX^1$. On the other hand, $X^1x=X^1ex=X^{1}xyx\subseteq X^1yx=X^1f=X^1e$ and $X^1e=X^1f=X^{1}yx\subseteq X^{1}x$, which implies $X^1x=X^1e$. Therefore, $x\in H_e\subseteq {\Uparrow}_{\!\pi} e$, which contradicts the choice of $x$. By analogy we can derive a contradiction from the assumption $yx\in {\Uparrow}_{\!\pi} e$.
Those contradictions show that ${\Uparrow}_{\!\pi} e$ is a prime coideal, equal to ${\Uparrow}e$.
\end{proof} 

Now we can prove that for every $e\in E(X)$ its $\two$-class ${\Updownarrow}e$ equals $\korin{\infty}{H_e}$. By Proposition~\ref{p:korin}, $\korin{\infty}{H_e}\subseteq{\Updownarrow}e$. To prove that $\korin{\infty}{H_e}={\Updownarrow}e$, choose any element $x\in{\Updownarrow}e$. Since $X$ is eventually Clifford, there exists an idempotent $f\in E(X)$ such that $x\in\korin{\infty}{H_f}$. Then there exists $n\in\IN$ such that $x^n\in H_f$ and hence $f\Updownarrow x^n\Updownarrow x\Updownarrow e$, see Proposition~\ref{p:quasi2}. By Claim~\ref{cl:Up},  $f\in{\Updownarrow}e\subseteq{\Uparrow}e=\{y\in X:e\le\pi(y)\}$ and hence $e\le f$. By analogy, $e\in {\Updownarrow}e={\Updownarrow}f\subseteq{\Uparrow}f=\{y\in X:f\le\pi(y)\}$ implies $f\le \pi(e)=e$. 
The inequalities $e\le f$ and $f\le e$ imply $e=f$ and hence $x\in\korin{\infty}{H_f}=\korin{\infty}{H_e}$, and finally, ${\Updownarrow}e=\korin{\infty}{H_e}$.
\end{proof}

Theorems~\ref{t:Esepar} and \ref{t:central} imply the following theorem describing  properties of $E$-hypercentral $\pi$-regular $E$-semigroups.

\begin{theorem}\label{t:EeClif} Every $E$-separated $\pi$-regular $E$-semigroup $X$ has the following properties.
\begin{enumerate}
\item $X$ is eventually Clifford and $E(X)$ is a semilattice.
\item For every idempotent $e\in E(X)$ we have ${\Updownarrow}e=\korin{\infty}{H_e}$ and ${\Uparrow}e=\{x\in X:e\le\pi(x)\}$.
\item For any idempotent $e\in E(X)$ the maximal subgroup $H_e\subseteq X$ is an ideal in the semigroup ${\Uparrow}e$.
\item For any $e\in E(X)$ and $x\in{\Uparrow}e$ we have $ex=xe$;
\item The map $\pi:X\to E(X)$ is a homomorphic retraction of $X$ onto $E(X)$.
\item The map $\hbar:X\to H(X)$, $\hbar:x\mapsto x\pi(x)$, is a homomorphic retraction of  $X$ onto its Clifford part $H(X)$.
\item The Clifford part $H(X)$ is a subsemigroup of $X$.
\end{enumerate}
\end{theorem}

\end{document}